\documentclass[11pt,reqno]{amsart}

\usepackage{amsbsy,amsfonts,amsmath,amssymb,amscd,amsthm,mathrsfs,verbatim,calc}
\usepackage{graphicx,epsfig,color}
\usepackage{enumitem}
\usepackage[a4paper,text={6.1in,8in},centering,includefoot,foot=1in]{geometry}

\small\normalsize
\theoremstyle{plain}
\newtheorem{mainthm}{Theorem}

\newtheorem{theorem}{Theorem}[section]

\newtheorem{proposition}[theorem]{Proposition}

\theoremstyle{definition}

\newtheorem{example}[theorem]{Example}

\numberwithin{equation}{section}
\let\oldmarginpar\marginpar
\renewcommand\marginpar[1]{\-\oldmarginpar[\raggedleft\footnotesize \textcolor{red}{#1}]{\raggedright\footnotesize\textcolor{red}{#1}}}

\begin{document}
\title{Attractor for minimal iterated function systems}
\author[A. Sarizadeh]{Aliasghar Sarizadeh}
\address{{Department of Mathematics,  Ilam University}}
\address{{Ilam, Iran.}}
\email{ali.sarizadeh@gmail.com}
\email{a.sarizadeh@ilam.ac.ir}

\begin{abstract}
In the present work, we study the attractors of iterated function systems (IFSs) on connected and compact metric spaces.
We prove that the whole of the phase space of a forward minimal IFS, for which some map admits an attracting fixed point,  
is an  attractor. 
\end{abstract}

\maketitle
\setcounter{tocdepth}{1}

\section{Introduction: preliminaries and results}
The contractivity of iterated function systems is a folkloric way to generate and explore a variety of self-similar fractals,  in the sense of Barnsley.
The base of the references on the self-similar  fractals is in the works by Hutchinson \cite{H81}, Mandelbrot  \cite{M82} and Barnsley \cite{B88}.

We begin with some notation and definitions. 
A dynamical system in the present paper is  a triple $(X,\Gamma, \varphi)$, where $X$ is a set, $\Gamma$ is a semigroup
and $\varphi:\Gamma \times X\to X$ is a map. Sometimes we write the dynamical system as a pair $(X,\Gamma )$.
In particular, let $\Gamma$ be a semigroup generated by a map $\phi : X \to X$. 
Take $\varphi : \Gamma \times X \to X$ so that $\varphi (\phi^n,x)=\phi^n(x)$, for every $n\in \mathbb{N}$. In this case,
$\Gamma = \{\phi^n\}_{n\in \mathbb{N}}$, the classical dynamical system $(X, \Gamma)$ is called 
a \textit{cascade system} and we use the standard notation: $(X, \phi)$. 

Let $\Phi^+$ be a semigroup generated by a finite set of maps, $\Phi$, on a set $X$. 
Consider a dynamical system $(X, \Phi^+, \varphi)$ with $\varphi:\Phi^+ \times X \to X$ 
so that $\varphi (\phi, x)=\phi(x)$. The dynamical system  $(X, \Phi^+, \varphi)$ 
is called an \textit{iterated function system} (IFS) associated to $\Phi$. We use the usual  notation: $(X, \Phi^+)$.

Throughout this paper, $X$ stands for a connected and compact metric space. We write $d$ for the metric on $X$.
We will assume that $\Phi$ consists of a finite collection $\{\phi_1,\ldots,\phi_\ell\}$ of  continuous maps on $X$.

A point $x$ is called a transitive point if $\overline{\mathcal{O}^+_{\Phi}(x)}=X$, where $\overline{A}$ is the closure of a subset $A$ of $X$ and $\mathcal{O}^+_{\Phi}(x)=\bigcup_{\varphi \in \Phi^+}\varphi(x)$.
Denote the set of all transitive points by Trans$(\Phi)$.
A subset $A$ of $X$ is called \textit{forward invariant} (resp. \textit{invariant}) for $(X, \Phi^+)$ or $\Phi$ if $\phi (A)\subset A$ (resp. $\phi(A)=A$), for every $\phi \in \Phi^+$.
IFS $(X, \Phi^+)$ is called \textit{forward minimal} if $X$ does not contain any non-empty, proper, closed subset which is forward invariant for $(X, \Phi^+)$. It is equivalent to Trans$(\Phi)=X$.

Associated to an IFS $(X,\Phi^+)$, for every $x_1,x_2\in X$, write
\begin{equation}\label{defmetric}
  d_\Phi(x_1,x_2):=\sup_{ \phi\in \Phi^+} d(\phi(x_1), \phi(x_2)).
\end{equation}
Clearly, $d_\Phi$ is a metric on $X$. 
A point $x$ is a sensitive point when the identity map $id_X:(X,d)\to (X,d_\Phi)$ is discontinuous at $x$.
An IFS $(X,\Phi^+)$ is sensitive when there exists $\epsilon >0$ such that every non-empty open subset has $d_\Phi$-diameter at least $\epsilon$.
A point $x$ which is not sensitive  will be called a equicontinuous  point i.e.
a point $x$ is an equicontinuous point when the identity map $id_X:(X,d)\to (X,d_\Phi)$ is continuous at $x$. 
An IFS $(X,\Phi^+)$ is equicontinuous when every point is an equicontinuous point and so $d_\Phi$ is a metric equivalent to $d$ on $X$.

The  Hutchinson operator $F:2^X\to 2^X$ associated with $\Phi$ is defined on the power set $2^X$ via 
$$F(Y):=\overline{\bigcup_{\phi\in \Phi}\phi(Y)};\ \ \ \ Y\subseteq X.$$
Write $\mathcal{K}(X)$ for the class of compact subsets of $X$.
The restriction $F:\mathcal{K}(X)\to \mathcal{K}(X)$ is well-defined.
For the Hutchinson operator $F$, define 
\[d_F(x_1,x_2):=\sup_{i\geq 0}d_H(F^i(x_1), F^i(x_2)),\] for every $x_1, x_2$.
Note that  $d_F$  is a metric on  $\mathcal{K}(X)$. 
By taking 
$$\Phi^k=\{\varphi_{k} \circ \dots \circ \varphi_{1} | \ \forall \ i=1,\dots,k; \ \varphi_{i}\in \Phi\},$$
one can write  $F^k(Y)=\overline{\bigcup_{\phi\in \Phi^k}\phi(Y)}$, for every subset $Y$ of $X$. For $\varphi \in \Phi^+$, 
let $|\varphi|$ denote the minimal natural number $n$ with
$\varphi \in \Phi^n$. 


A non-empty compact invariant set $A$ of $X$ is called an attractor of IFS $(X, \Phi^+)$  
 if for every compact subset $B$ of $X$ in the Hausdorff metric
\begin{equation}\label{3}
 F^j(B) \to A \text{ \ as \ } j \to \infty.
 \end{equation}
An attractor is this way uniquely defined. 
An IFS with an attractor is called an asymptotically stable system.
A local version of attractor is called strict attractor. To be more precise, a non-empty closed subset $A$ of $X$ is a strict attractor of IFS $(X, \Phi^+)$  
if there exists an open neighborhood $U(A) \supset A$ such that in the Hausdorff metric
\begin{equation}\label{3}
 F^j(B) \to A \text{ \ as \ } j \to \infty,
 \end{equation}
for every compact subset $B$ of $U(A)$.
The basin $B(A)$ of an attractor $A$ is the union of all open neighborhoods $U$ for which \ref{3} holds. 

One would like simple conditions that imply the existence of a strict attractor. 
One possible criterion  is in   \cite{bgms}.
An IFS generated by a family $\Phi$ of continuous maps of $X$ is
quasi-symmetric if there is $\phi \in \Phi^+$ so that its inverse map $\phi^{-1} \in \Phi^+$. 
In \cite{bgms} it is shown that in case IFS $(X, \Phi^+)$ is forward minimal and $\Phi$ is  quasi-symmetric,
then $X$ is an attractor.
Theorem \ref{attfixed}(\ref{p1}) below provides a different and elementary condition that implies that the whole phase space is an attractor.

\begin{mainthm}\label{attfixed}
Let   $\Phi=\{\phi_1,\dots,\phi_\ell\}$ be a finite set of continuous maps on $X$ and IFS $(X, \Phi^+)$ be forward minimal. 
Assume there exists $\phi: X \to X$ belonging to $\Phi$ with an attracting fixed point $p$. 
 Then the followings hold:
\begin{enumerate}
\item\label{p1} The phase space $X$ is an attractor of  IFS $(X, \Phi^+)$.
\item\label{p2} The system $(\mathcal{K}(X),F)$ is equicontinuous.
\end{enumerate}
\end{mainthm}
As a consequence of equicontinuity of the Hutchinson operators, we study the shadowing property (sometimes called 
pseudo orbit tracing property) of the Hutchinson operator on the  hyperspace $\mathcal{K}(X)$. 
We will need to consider the $i$-th iteration of the map
$F:\mathcal{K}(X)\to \mathcal{K}(X)$ as $F^{i+1} = F^i \circ F$, where $F^0$
is  the identity map. An orbit of the system $(\mathcal{K}(X),F)$  is a sequence
$\{X_i\}_{i=0}^\infty$ with $X_{i+1} = F(X_i)$, for all $i\in \mathbb{N}$. We say that a sequence
$\xi=\{X_i\}_{i=0}^\infty \subset 2^X$ is a $\delta$-pseudo-orbit if $d_H(X_{i+1}, F(X_i)) < \delta$, for all $i \in\mathbb{N}$.
The Hutchinson operator $F:\mathcal{K}(X)\to \mathcal{K}(X)$  has the shadowing property
if for
each $\epsilon > 0$ there exists $\delta>0$ such that for any $\delta$-pseudo-orbit $\xi = \{X_i\}$ one can find a set
$Y$ in $ \mathcal{K}(X)$ such that $d_H(X_i,F^i(Y))<\epsilon$, for all $i\in \mathbb{N}$. 
In \cite{fg}, it is shown that a cascade system  $(X,f)$ has the 
shadowing property
if and only if the cascade system generated by the induced map
$(\mathcal{K}(X),2^f)$  has the 
shadowing property.
In \cite{bie99}, the author showed that the Hutchinson operator of a given weak contracting
IFS on a compact metric space has the 
shadowing property.
The base of the argument
of the main result in \cite{bie99} comes from the following fact: the Hutchinson operator of weak
contracting IFSs has a unique fixed point \cite{bie95}. 

\begin{mainthm}\label{shadowing}
Let   $\Phi=\{\phi_1,\dots,\phi_\ell\}$ be a finite set of continuous maps on $X$. 
Suppose $X$ is an attractor of IFS $(X, \Phi^+)$ and  the Hutchinson operator $F$ 
is equicontinuous on the hyperspace  $\mathcal{K}(X)$.
Then the IFS  $(X, \Phi^+)$ has the  shadowing property.
\end{mainthm}

%
%
%
%
%
%
%
%
%
%
%
%
%
%
%
%
%
%
%
%
\section{Attractors for forward minimal IFSs}
%
%
%
%
%


In this section we prove Theorem~\ref{attfixed}.
As a first step, we make the following observation which is routine
and so we do not provide a proof.

\begin{proposition}\label{4.6} 
If $x$ is an equicontinuous point of an IFS $(X, \Phi^+)$, then $\{x\}$ is an equicontinuous 
point of $(\mathcal{K}(X), F)$. 
\end{proposition}

The converse statement is not necessarily true. 
In the following example we give an IFS $(X, \Phi^+)$ with non-equicontinuous 
points but the cascade system   $(\mathcal{K}(X), F)$ is equicontinuous.

\begin{example}\label{4.7} 
	Suppose that $\Phi=\{f_1,f_2,f_3\}$ and $\Psi=\{f_1,f_2\}$ where $f_1, f_2, f_3:  [0, 1] \to [0, 1]$ are given as in 
	Figure 1. Let $f_\Phi$ and $F_\Psi$ denote the  Hutchinson operators for IFS $([0,1],\Phi^+)$  and IFS $([0,1],\Psi^+)$, respectively.
	In \cite{ip}, the authors 
	proved that IFS $([0, 1],\Phi^+)$ is a forward minimal IFS with equicontinuous and non-equicontinuous 
	points. Indeed, this system is neither sensitive nor equicontinuous. 
	\begin{figure*}[!h]
		\centering
		{\label{ali}
			 \includegraphics[width=0.4\textwidth]{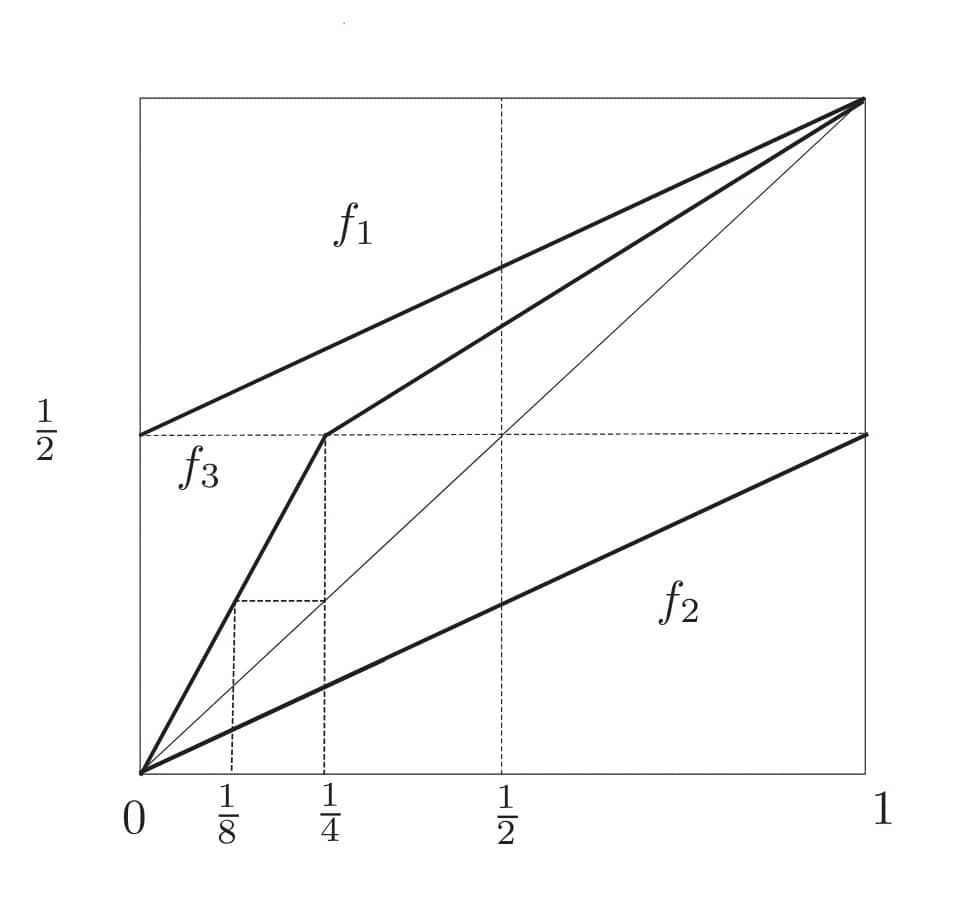}
		}
		\caption{ }
	\end{figure*}
	As is well known, in the sense of Barnsley, $I=[0,1]$ is 
	a fractal of the IFS $(I,\Psi^+)$ and so $F_\Psi^n(x)\to I$, in the Hausdorff metric, for every $x\in I$.
	It is not hard to see that for every $\epsilon>0$ one can chose a natural number $n_0$ so that
	$d_H(F^i_\Psi(x),I)<\dfrac{\epsilon}{2}$, for every $i>n_0$ and every $x\in I$.
	Since $\Psi$ contains of contraction maps, for $\epsilon>0$, we have
	$$d_{F_\Psi}(x,y)=\sup_{i\in \mathbb{N}_0}d_H(F^i_\Psi(x),F^i_\Psi(y))<\epsilon;\ \ \  \forall y\in B(x,\dfrac{\epsilon}{2}).$$
	Clearly,  $F^i_\Psi(x)\subset F^i_\Phi(x)\subset I$ for any $x\in I$ and any $i\in \mathbb{N}$. 
	Then
	$d_H(F^i_\Phi(x),I)<\dfrac{\epsilon}{2}$ for any $x\in X$ and any $i>n_0$.
	It implies that for every $\epsilon>0$ there exists $\delta>0$ so that
	$$d_{F_\Phi}(x,y)=\sup_{i\in \mathbb{N}_0}d_H(F^i_\Phi(x),F^i_\Phi(y))<\epsilon;\ \ \  \forall y\in B(x,\delta).$$
	It means that   $\{x\}$ is an equicontinuous point of $(\mathcal{K}(X), F_\Phi)$, for every $x\in I$. 
\end{example}

%



\begin{proof}[Proof of Theorem~\ref{attfixed} (\ref{p1})]
Without loss of generality, let $\phi_1$ belonging to $\Phi$ be the map admitting an attracting fixed point $p$.
Take a neighborhood  $U$ of $p$ such that $\phi_1(\overline{U})\subset U$ and    $\bigcap_{n\geq0} \phi_1^n(U)=\{p\}$.
Since the IFS $(X, \Phi^+)$ is forward minimal, there exist $\varphi_1,\dots,\varphi_k$ in $\Phi^+$ so that 
\begin{equation}\label{minimalpro}
\bigcup_{i=1}^k \varphi_i^{-1}(U)=X.
\end{equation}
Hence, for every $x\in X$, one can choose $i$ in $\{1,\dots,k\}$ so that $\varphi_i(x)\in U$.
So, for every $x\in X$ 
$$
F^n(x) \cap U\neq \emptyset; \ \ \ \ \forall \ n\geqslant k_U:=\max\{|\varphi_i|:\ i=1,\dots,k\}.
$$

Take  $x_0$ and $y_0$ in $X$ and $\epsilon>0$.
$F^n(x_0) \cap U\neq \emptyset$ for every $ n\geqslant k_U$.
For some  $i\in\{1,\dots,k\}$ with
$\varphi_{i}(x_0)\in U$, we have
\begin{equation}\label{att12}
\phi_1^n(\varphi_{i}(x_0))\to p\ \ \ \ \mathrm{as}\ \ n\to\infty.
\end{equation}
By the forward minimality of IFS $(X, \Phi^+)$  and continuity of its generators, 
one can find $\delta>0$ and   $\varphi'_0\in \Phi^+$ so that $B(p,\delta)\subseteq U$ and  $\varphi'_0(B(p,\delta))\subseteq B(y_0,\epsilon/2)$.
By (\ref{att12}), there exists an integer $n_0$ so that 
\[
\phi_1^n(\varphi_{i}(x_0)) \in B(p,\delta),
\] 
for all $n\geq n_0$.
Therefore, 
\[
\varphi'_0(\phi_1^n(\varphi_{i}(x_0)))\in B(y_0,\epsilon/2),
\] 
for all $n\geq n_0$.
It follows that 
\begin{equation}\label{sh2}
F^{j}(x_0)\cap B(y_0,\epsilon/2)\neq \emptyset;\ \ \ \forall\  \ j\geq |\varphi_i|+n_0+|\varphi'_0|.
\end{equation}
Furthermore, for every $x$ in $X$, one has
\begin{equation}\label{sh3}
F^{j}(x)\cap B(y_0,\epsilon/2)\neq \emptyset; \ \ \ \forall\ \ j\geq k_U+n_0+|\varphi'_0|.
\end{equation}
Since $X$ is compact,  by using the mentioned notation and argument,  one can find points $y_0,y_1,\dots, y_s$ in $X$ 
and $\varphi'_0,\varphi'_1,\dots, \varphi'_s$ in $\Phi^+$ so that the following hold:
\begin{itemize}
	\item $\bigcup_{i=0}^s B(y_i,\epsilon/2)=X$.\\
	\item $F^{j}(x_0)\cap B(y_i,\epsilon/2)\neq \emptyset;\ \ \forall\ i=0,1,\dots s \ \wedge \ \forall\  j\geq k_U+n_0+|\varphi'_i|.$
\end{itemize}
Observe that 
\begin{equation}\label{forallx}
	d_H(F^j(x),X)<\epsilon/2;\ \ \ \forall\ x\in X \ \wedge\ \forall\ j\geq k_U+n_0+s_\epsilon,
\end{equation}
where $s_\epsilon=\max\{|\varphi'_i|:\ i=0,1,\dots,s\}$.
\end{proof}

\begin{proof}[Proof of Theorem~\ref{attfixed} (\ref{p2})]
By the argument for part (\ref{p1}), since $\varphi_i(x_0)\in U$ for some $i$ and $\phi_1,\varphi_{i}$ are continuous, 
one can find $\delta_0>0$ and an integer $n_*$ so that 
$ \phi_1^{n}(\varphi_{i}(B(x_0,\delta_0)))\subset B(p,\delta)\subset U$, for $n\geq n_*$.

 Again, by applying  the argument of part (\ref{p1})
 on every $y\in B(x_0,\delta_0)$,
we have  
\begin{equation*}
d_H(F^j(y),X)<\epsilon/2; \ \ \ j\geq k_U+n_*+s_\epsilon.
\end{equation*}
So, it is clear that
\begin{equation}\label{equico}
\sup\{d_H(F^j(x),F^j(y)):\ x,y\in B(x_0,\delta_0)\ \ \wedge\ \  j\geq k_U+n_*+s_\epsilon\}<\epsilon.
\end{equation}
By continuity of generators for some 
$\delta_{x_0}<\delta_0$, we have
\begin{equation}
\max\{d_H(F^j(x),F^j(y)):\ x,y\in B(x_0,\delta_{x_0})\ \ \wedge\ \  j\leq k_U+n_*+s_\epsilon\}<\epsilon.
\end{equation}
Together with (\ref{equico}), one can conclude equicontinuity of the system $(\mathcal{K}(X),F)$ restricted to the set 
$\{\{x\}|\ x\in X\}$.
 
 To continue the proof, suppose that $D$ is an arbitrary element of $\mathcal{K}(X)$. 
 By continuity of generators, there exists $\delta_D>0$ so that  
$$
\sup_{j\leq k_{U}+n_0+s_\epsilon}d_H(F^j(D),F^j(D'))<\epsilon;\ \ \ \forall\ D'\in \mathcal{K}(X)\cap B_{d_{_{H}}}(D,\delta_D).
$$
By (\ref{forallx}), $d_H(F^j(x),X)<\frac{\epsilon}{2}$, for all $x\in D\cup D'$ and all $j\geq k_{U}+n_0+s_\epsilon$.
 Thus, 
$$\sup_{j\geq k_{U}+n_0+s_\epsilon}d_H(F^j(D),F^j(D'))<\epsilon;\ \ \ \forall\ D'\in \mathcal{K}(X).
$$
 Hence, 
$$
\sup_{i\in \mathbb{N}}d_H(F^i(D),F^i(D'))<\epsilon;\ \ D'\in \mathcal{K}(X)\cap B_{d_H}(D,\delta_D),
$$
which completes the proof of part (\ref{p2}).
\end{proof}

\begin{example}\label{4.8} 
	Suppose $g_1$ is a north-south pole diffeomorphism of the circle $S^1$
	possessing an attracting 
	fixed point $p$ as a north pole with multipliers $1/2 < g'_1(p) < 1$ and a repelling fixed point $q$ as south 
	pole. Consider the map $g_2 = R_\alpha$ where $R_\alpha$ is the rotation by irrational angle $\alpha$ on the circle. 
	Write $\Phi=\{g_1,g_1^{-1},g_2,g_2^{-1}\}$ and $\Psi=\{g_1,g_2\}$. 
	
	Both of IFSs  $(S^1,\Phi^+)$ and $(S^1,\Psi^+)$ are forward minimal and sensitive. 
    By  following \cite{BFS}, the IFS  $(S^1,\Phi^+)$ is quasi-symmetric, and so  the circle $S^1$
	is an attractor of IFS $(S^1, \Phi^+)$.
    Since $g_1$ admitting an attracting fixed point, Theorem~\ref{attfixed} implies that  the cascade system $(\mathcal{K}(S^1), F_\Psi)$ is equicontinuous 
	on the set $\mathcal{K}(X)$ 
	and also the circle $S^1$
	is an attractor of IFS $(S^1, \Psi^+)$. 
\end{example}
%
%
%
%



%
%
%

\section{Approximation of attractors by the Hutchinson shadowing property}
We start  with approximation of the whole  space $X$ (our attractor) by pseudo orbits of the Hutchinson operator of an IFS $(X, \Phi^+)$.
\begin{proof}[Proof of Theorem~\ref{shadowing} ]

	As above, 
	by compactness of $X$, we have that 
	for every closed subset $Y$ of $X$,
	$d_H(F^i(Y),X)<\frac{\epsilon}{4}$,  for every $i\geq n_{_{\frac{\epsilon}{4}}}$.
	
	On the other hand, according to Theorem 5  of \cite{zac}, for an arbitrary natural number $n$ there exist 
	positive real numbers $\delta_n > 0$ and $\gamma_n> 0$ such that every $\delta_n$-pseudo-orbit 
	$\{X_i\}^n_{i=0}$ in $\mathcal{K}(X)$ is 
	$\frac{\epsilon}{4}$-shadowed by the orbit $\{F^i(Y)\}^n_{i=0}$ for an arbitrary 
	$Y\in B(X_0,\gamma_n)\subseteq \mathcal{K}(X).$
	
	These yield that for every $\delta_n$-pseudo-orbit 
	$\{X_i\}^\infty_{i=0}$ in $\mathcal{K}(X)$, with $n>n_{_{\frac{\epsilon}{4}},}$ by inductive process  one can prove that
	\begin{equation}\label{xxi}
		\sup_{i\geq n_{_{\frac{\epsilon}{4}}}}d_H(X,X_i)<\frac{\epsilon}{2}.
	\end{equation}
	Indeed, for every $k\in \mathbb{N}\cup \{0\}$ and finite subsequence $\{X_i\}_{i=k}^{k+n}$ of  
	the $\delta_n$-pseudo-orbit $\{X_i\}_{i=0}^\infty$, we can find $Y_k$ with $d_H(X_k,Y_k)<\gamma_n$ so that 
	$$
	d_H(F^j(Y_k),X_{j+k})<\frac{\epsilon}{4};\ \ \ \forall\ j=0,1,\dots,n.
	$$
	Moreover,  for  $j+k>n_{_{\frac{\epsilon}{4}}}$, $d_H(F^j(Y_k),X)<\frac{\epsilon}{4}$.
	Thus, for  $j+k>n_{_{\frac{\epsilon}{4}}}$, $d_H(X,X_{j+k})<\frac{\epsilon}{2}$.
	By induction on $k$, one can prove the inequality \ref{xxi}.
	
	So,  it is easy to see that for every $\delta_n$-pseudo-orbit 
	$\{X_i\}^\infty_{i=0}$ in $\mathcal{K}(X)$ one has 
	$$\sup_{0 \leq i\leq \infty}d_H(F^i(Y),X_i)<\epsilon.$$
\end{proof}

  an example in which $F$ is  equicontinuos but $F^{-1}$ is sensitive.
\begin{example}
Let us equipped the product space  $\Sigma^+_2=\{1,2\}^\mathbb{N}$
to the metric $d(\eta, \omega) =\inf\{2^{1-n}: \ \eta_i = \omega_i,\ \forall \ i<n\}$. 
For $j = 1, 2$, let $(\phi_j(\eta))_1 = j$ and $(\phi_j(\eta))_i = \eta_{i-1}$ for $i\geq 2$. 
Take $\Phi = \{\phi_1, \phi_2\}$. In this case $F^{-1}$ is a map and it is the shift map given by 
$(F^{-1}(\eta))_i = (\sigma(\eta)) _i=\eta_{i+1}$, for every $i$. Thus, the periodic points of $F^{-1}$
is dense and so it is sensitive. However, the $F$  is equicontinuous.
\end{example}

%
%
%
%
%
%
%
%
%
%
%
%
%
%
%
%
%
%
%
%
%
%


\textbf{Acknowledgments.} We thank Ale Jan Homburg for useful discussions and suggestions.

\end{document}